\documentclass[a4paper,10pt]{amsart}

\usepackage[T1]{fontenc} \usepackage[utf8]{inputenc}
\usepackage{lmodern}

\usepackage{enumerate}

\usepackage{amsfonts,amsmath,amstext,amsbsy,amssymb}
\usepackage{amsbsy,amsopn,amsthm,amscd,amsxtra}

\usepackage{latexsym,mathrsfs}

\usepackage{mathabx}

\usepackage{hyperref}
\RequirePackage[dvipsnames]{xcolor} % [dvipsnames]
\definecolor{halfgray}{gray}{0.55} 
\definecolor{webgreen}{rgb}{0,0.5,0}
\definecolor{webbrown}{rgb}{.6,0,0} \hypersetup{%
  colorlinks=true, linktocpage=true, pdfstartpage=3,
  pdfstartview=FitV,%
  breaklinks=true, pdfpagemode=UseNone, pageanchor=true,
  pdfpagemode=UseOutlines,%
  plainpages=false, bookmarksnumbered, bookmarksopen=true,
  bookmarksopenlevel=1,%
  hypertexnames=true,
  pdfhighlight=/O,%hyperfootnotes=true,%nesting=true,%frenchlinks,%
  urlcolor=webbrown, linkcolor=RoyalBlue,
  citecolor=webgreen, %pagecolor=RoyalBlue,%
  pdftitle={},%
  pdfauthor={},%
  pdfsubject={2000 MAthematical Subject Classification: Primary:},%
  pdfkeywords={},%
  pdfcreator={pdfLaTeX},%
  pdfproducer={LaTeX with hyperref}%
}

\newcommand{\abs}[1]{\left\lvert{#1}\right\rvert}
\newcommand{\norm}[1]{\left\|{#1}\right\|}

\newcommand{\R}{\mathbb{R}}\newcommand{\N}{\mathbb{N}}
\newcommand{\Z}{\mathbb{Z}}

 \newcommand{\ie}{i.e.\ }

\newtheorem*{teo*}{Theorem}

\newtheorem{theorem}{Theorem}[section]
\newtheorem{proposition}[theorem]{Proposition}
\newtheorem{lemma}[theorem]{Lemma}

\theoremstyle{definition} \newtheorem{definition}[theorem]{Definition}

\theoremstyle{remark} \newtheorem{remark}[theorem]{Remark}

\newcommand{\Diff}[1]{\mathrm{Diff}^{#1}}

\newcommand{\eps}{\varepsilon} \newcommand{\carr}{\righttoleftarrow}

 \DeclareMathOperator{\Lip}{Lip}
 \DeclareMathOperator{\Fix}{Fix}
 \DeclareMathOperator{\GL}{GL}

\title[Cohomology of diffeomoprhism-valued cocycles]{Cohomology of
  dominated diffeomorphism-valued cocycles over hyperbolic systems}

\author[L. Backes]{Lucas H. Backes} \address{IMPA, Estrada Dona
  Castorina 110, 2460-320 Rio de Janeiro, RJ, Brazil.}
\email{lhbackes@impa.br}

\author[A. Kocsard]{Alejandro Kocsard} \address{IME - Universidade
  Federal Fluminense. Rua M\'ario Santos Braga S/N, 24020-140
  Niter\'oi, RJ, Brazil}\urladdr{www.professores.uff.br/kocsard}
\email{akocsard@id.uff.br}

\thanks{} \date{\today}

\begin{document}

\maketitle

\begin{abstract}
  We prove a rigidity theorem for dominated Hölder cocycles with
  values on diffeomorphism groups of a compact manifold over
  hyperbolic homeomorphisms. More precisely, we show that if two such
  cocycles have equal periodic data, then they are cohomologous.
\end{abstract}

\maketitle

\section{Introduction}
\label{sec:intro}

Given a homeomorphism $f\colon M\carr$ and a topological group $G$, a
\emph{$G$-cocycle over $f$} is a continuous map $\alpha\colon\Z\times
M\to G$ satisfying
\begin{equation}
  \label{eq:cocycle-prop}
  \alpha^{(m+n)}(x)=\alpha^{(m)}\big(f^n(x)\big)\alpha^{(n)}(x),
  \quad\forall m,n\in\Z,\ \forall x\in M.
\end{equation}

Two cocycles $\alpha$ and $\beta$ over $f$ are said to be
\emph{cohomologous} whenever there exists a continuous map $P\colon
M\to G$, usually called \emph{transfer map}, such that
\begin{displaymath}
  \alpha^{(n)}(x)=P\big(f^n(x)\big)\beta^{(n)}(x)P(x)^{-1},
  \quad\forall n\in\Z, \forall x\in M.
\end{displaymath}
As it can be easily verified, the cohomology relation is an
equivalence one over the space of (continuous) cocycles.

In the particular case that $\beta$ is the constant function equal to
$e_G$ (where $e_G$ denotes the identity element of $G$), we say that
$\alpha$ is a \emph{coboundary.}

Many important questions in dynamical systems can be reduced to the
problem of determine whether certain cocycles are cohomologous. So, it
is not surprising that this equivalence relation has been extensively
studied in last decades (see for instance \cite{KatokRobinson} for a
survey).

As it had been already noticed in \cite{KocCohomC0Inst}, in general
the $C^0$-category is not the right one to study cohomology of
dynamical systems, and depending on the dynamical properties of the
system, different degrees of regularity are required to guaranty a
rather reasonable description of cohomology classes.

When $f$ is a hyperbolic homeomorphism (see \S~\ref{sec:preelim} for
precise definitions), it seems like Hölder-regularity (for cocycles)
is enough to study its cohomology. In fact, the seminal works of Liv\v
sic~\cite{LivCertPropHomol,LivsicCohomDynSys} claim that, when $G$
admits a compatible bi-invariant complete distance, the \emph{periodic
  data} of $f$ completely characterizes Hölder-continuous
$G$-coboundaries. More precisely, a Hölder $G$-cocycle $\alpha$ over
$f$ is a coboundary if and only if
\begin{displaymath}
  \alpha^{(n)}(p)=e_G,\quad\forall p\in\Fix(f^n).
\end{displaymath}

This result has been considerably extended to more general groups and
to higher regularity by several different authors
\cite{delaLlaveMarcoMoriyon, NiticaTorokRegularResultsSolLiv,
  NicolPollicottLivSemisimpleLieGr, delaLlaveWindsorLivThmNonComm,
  KalininLivThmMat, KocPotrieLivDiffCirc}, and such results are
usually called \emph{Liv\v{s}ic type theorems} (see also
\cite{NiticaKatokRigHighRankI} and references therein).

More generally, inspired by this results, it would be natural to ask
whether two Hölder $G$-cocycles $\alpha $ and $\beta$ over a
hyperbolic system $f$ satisfying the so called \emph{periodic orbit
  condition}, \emph{POC} for short,
\begin{equation}
  \label{eq:POC}\tag{POC}
  \alpha^{(n)}(p)=\beta^{(n)}(p),\quad\forall n\in\Z,\ \forall
  p\in\Fix(f^n), 
\end{equation}
are indeed cohomologous.

When $G$ is an Abelian group, the space of $G$-cocycles is naturally
an Abelian group itself, and coboundaries form a subgroup. Hence,
under the hypothesis of commutativity, since it can be easily verified
that two cocycles are cohomologous if and only if their difference is
a coboundary, the problems of characterization of coboundaries and
cohomology classes are completely analogous.

However, when $G$ is non-Abelian, the second problem is much more
complicated than the first one.

In the present paper we study the characterization of cohomology
classes for cocycles with values in the group $G=\Diff{r}(N)$ (\ie the
group of $C^r$-diffeomorphisms of a compact manifold $N$) and
satisfying a domination condition. The domination condition is used to
construct invariant holonomies which plays a key role in our
arguments.

This generalizes the previous result independently obtained by the
first author~\cite{BackesRigFibBunchCoc} and
Sadovskaya~\cite{SadovskayaCohomFibBunchCoc} for linear cocycles (\ie
where $G=\GL_d(\R)$) satisfying a fiber bunching condition. Similar
results have been previously gotten by
Parry~\cite{ParryLivPerPointThm}, for cocycles adimiting a
bi-invariant distance, a by Schmidt~\cite{SchmidtRemLivThNonAbCoc} for
cocycles satisfying a ``bounded distortion condition'', which is much
stronger than our domination assumption.

These paper is organized as follows: in section~\ref{sec:main-results}
we present the precise statement of the main result. In
section~\ref{sec:preelim} we give some definitions and preelimiarie
results. Section~\ref{sec:diff-grp-top} is devoted to describe a
distance in $\Diff{r}(N)$ and some auxiliary results. In
section~\ref{sec:inv-holonomies} we build the main tool used in our
proof: invariant holonomies. In sections~\ref{sec:constr-trans-map}
and \ref{sec:concluding-the-proof} we construct, in some sense,
explicitly the transfer map while section
\ref{sec:improving-regularity} is devoted to improve the regularity of
these transfer map.

\subsection{Main results}
\label{sec:main-results}

The main result of this work is the following one (see
\S~\ref{sec:preelim} for precise definitions):

\begin{theorem}
  \label{thm:main-1} Let $f\colon M\carr$ be a Lipschitz continuous
  transitive hyperbolic homeomorphism on a compact metric space
  $(M,d)$, $N$ a compact smooth manifold and
  $\alpha,\beta\colon\Z\times M\to\Diff r(N)$ two $(2r-1)$-dominated
  $\nu$-H\"older cocycles over $f$ satisfying the \emph{periodic data
    condition:}
  \begin{equation}
    \label{eq:per-data-cond}
    \alpha^{(n)}(p)=\beta^{(n)}(p), \quad\forall n\in Z,\ \forall
    p\in\Fix(f^n). 
  \end{equation}

  Then, there exists a $\nu$-Hölder continuous map $P\colon
  M\to\Diff{r-4}(N)$ such that
  \begin{equation}
    \label{eq:cohomo-eq-main-thm}
    \alpha^{(n)}(x)=P\big(f^n(x)\big)\circ\beta^{(n)}(x)\circ
    P(x)^{-1}, \quad\forall x\in M,\ \forall n\in\Z.
  \end{equation}
  Moreover, $P(M)\subset\Diff{r-1}(N)$ and if $f$ is a $C^{r-1}$
  Anosov diffeomorphism and the cocycles $\alpha$ and $\beta$ are
  $C^{r-1}$ then $P$ is $C^{r-1-\varepsilon}$ for any small
  $\varepsilon >0$.
\end{theorem}

At this point we would like to remark that by $P\colon
M\to\Diff{r-4}(N)$ being a $\nu$-H\"older continuous map we mean that
$P$ is $\nu$-H\"older with respect to the distance function $d_{r-4}$
given by \eqref{eq:dr-distance-def} and the cocycle $\alpha
:\mathbb{Z}\times M \to \Diff {r}(N)$ to be $C^{r-1}$ means that the
induced map $\alpha^{(1)}\colon M\times N \to N$ is $C^{r-1}$.

The strategy of proving this Theorem is closely related to the proof
of an analogous rigidity theorem for linear cocycles, proved by the
first author in \cite{BackesRigFibBunchCoc}. However, here we have
additional difficulties due to the fact that the group of
diffeomorphisms is infinite-dimensional.

The $(2r-1)$-domination condition is used to guarantee the existence
of invariant holonomies with good regularity. Using this invariant
holonomies we ``explicitly'' construct the transfer map on a dense
subset where we prove it is $\nu$-Hölder. The last step consists in
extending $P$ to the whole space $M$ and to prove that it has, at
least in some cases, a better regularity.

As a consequence of our methods we also obtain the following
\begin{theorem}
  \label{thm:main-2}
  Let $M$, $N$ and $f$ be as in Theorem~\ref{thm:main-1} and $\alpha
  ,\beta\colon\Z\times M\to\Diff r(N)$ be two $(2r-1)$-dominated
  $\nu$-H\"older cocycles over $f$. Let us assume there exists
  $P\colon M\to\Diff 1(N)$ $\nu$-H\"older such that
  \begin{displaymath}
    \alpha^{(1)}(x)=P(f(x))\circ\beta^{(1)}(x)\circ P(x)^{-1},
    \quad\forall x\in M,
  \end{displaymath}
  and assume also that there exists $x_0\in M$ such that
  $P(x_0)\in\Diff{r-1}(N)$. Then,
  $P(M)\subset\Diff{r-1}(N)$. Moreover, if $f$ is a $C^{r-1}$ Anosov
  diffeomorphism and the cocycles $\alpha$ and $\beta$ are $C^{r-1}$,
  then $P$ is $C^{r-1-\varepsilon}$, for any $\varepsilon >0$.
\end{theorem}

\section{Preeliminaries and notations}
\label{sec:preelim}

\subsection{Hyperbolic homeomorphisms}
\label{sec:hyperb-homeo}

Let $(M,d)$ be a compact metric space and $f\colon M\carr $ be a
homeomorphism. Given any $x\in M$ and $\varepsilon >0$, define the
\emph{local stable} and \emph{unstable sets} by
\begin{align*}
  W^s_\eps(x,f) &:= \left\{y\in M : d(f^n(x),f^n(y))\leq\eps,\ \forall
    n \geq 0\right\}, \\
  W^u_\eps(x,f) &:= \left\{y\in M : d(f^n(x),f^n(y))\leq\eps,\ \forall
    n \leq 0\right\},
\end{align*}
respectively. Where there is no risk of ambiguity, we just write
$W^s_\varepsilon(x)$ instead of $W^s_\varepsilon(x,f)$, and the same
holds for the local unstable set.

Following \cite{AvilaVianaExtLyapInvPrin}, we introduce the following

\begin{definition}
  \label{def:hyperbolic-homeo}
  A homeomorphism $f\colon M\carr$ is said to be \emph{hyperbolic with
    local product structure} (or just \emph{hyperbolic} for short)
  whenever there exist constants $C_1,\eps,\tau>0$ and $\lambda\in
  (0,1)$ such that the following conditions are satisfied:
  \begin{itemize}
  \item[(h1)] $d(f^n(y_1),f^n(y_2)) \leq C_1\lambda^n d(y_1,y_2)$,
    $\forall x\in M$, $\forall y_1,y_2 \in W^s_\eps (x)$, $\forall
    n\geq 0$;
  \item[(h2)] $d(f^{-n}(y_1), f^{-n}(y_2)) \leq C_1\lambda^n
    d(y_1,y_2)$, $\forall x\in M$, $\forall y_1,y_2 \in W^u_\eps (x)$,
    $\forall n\geq 0$;
  \item[(h3)] If $d(x,y)\leq\tau$, then $W^s_\eps(x)$ and
    $W^u_\eps(y)$ intersect in a unique point which is denoted by
    $[x,y]$ and depends continuously on $x$ and $y$.
  \end{itemize}
\end{definition}

For such homeomorphisms, one can define the \emph{stable} and
\emph{unstable sets} by
\begin{displaymath}
  W^s(x,f):= \bigcup_{n\geq 0} f^{-n}\big(W^s_\varepsilon(f^n(x))\big)
  \quad\text{and}\quad W^u(x,f):= \bigcup_{n\geq 0} 
  f^{n}\big(W^u_\varepsilon(f^{-n}(x))\big),
\end{displaymath}
respectively.

Notice that shifts of finite type and basic pieces of Axiom A
diffeomorphisms are particular examples of hyperbolic homeomorphisms
with local product structure (see for instance \cite[Chapter IV,
\S~9]{ManeBook} for details).

\subsection{Cocycles and cohomology}
\label{sec:cocycles-cohomology}

Given any homeomorphism $f\colon M\carr$ and a topological group $G$,
a \emph{$G$-cocycle} over $f$ is a continuous map
$\alpha\colon\Z\times M\to G$ such that
\begin{displaymath}
  \alpha^{(m+n)}(x)=\alpha^{(m)}\big(f^n(x)\big)\alpha^{(n)}(x),
  \quad\forall m,n\in\Z,\ \forall x\in M.
\end{displaymath}

On the other hand, if $H$ another topological group such that
$G\subset H$ just as groups (\ie their topologies are not necessarily
related), then two $G$-cocycles $\alpha$ and $\beta$ over $f$ are said
to be \emph{$H$-cohomologous} whenever there exists a continuous map
$P\colon M\to H$ such that
\begin{displaymath}
  \alpha^{(n)}(x)=P\big(f^n(x)\big)\beta^{(n)}(x)P(x)^{-1},
  \quad\forall n\in\Z,\ \forall x\in M.
\end{displaymath}

\subsection{Diffeomorphism groups and their topologies}
\label{sec:diff-grp-top}

In this subsection we recall some concepts about $C^r$-topologies on
groups of diffeomorphisms and to do that, we mainly follow \cite[\S
5.2]{delaLlaveWindsorLivThmNonComm}. From now on, $N$ will denote a
compact smooth Riemannian manifold.  We write $\Diff{r}(N)$ for the
group of $C^r$-diffeomorphisms of $N$.

It is well known that $\Diff r(N)$ has a Banach manifold structure
modeled on the space of $C^r$ vector fields with local charts given by
composition with the the exponential map.

The topology induced by this Banach manifold structure coincides with
the one induced by the \emph{$C^r$-distance} $d_r$ which is defined as
follows:

Let us consider a smooth curve $p:\R\to\Diff r(N)$. Then, for any
$t\in\R$ and $y\in N$ there exists a neighborhood $V$ of $t$ and $U$
of $0\in U\subset T_yN$ such that for any $s\in V$ the local
representative $\tilde{p}(s)_y:U\subset T_yN\to T_{p(t)(y)} $ is
defined by
\begin{displaymath}
  p(s)(\exp _yv)=\exp _{p(t)(y)}(\tilde{p}(s)_y(v)).
\end{displaymath}
Using the standard identification of the tangent space to a linear
space with the linear space itself, we obtain
\begin{displaymath}
  D^n\tilde{p}(s)_y\colon T_yN^{\otimes n}\rightarrow T_{p(t)(y)}N.
\end{displaymath}

Now, since this is a curve in a fixed linear space, we can
differentiate it with respect to $s$ to get
\begin{displaymath}
  \frac{d}{dt}D^n_yp(t):=\frac{d}{ds}D^n\tilde{p}(s)_y\Big|_{s=t}(0).  
\end{displaymath}

Now, consider a path $p\colon[0,1]\to\Diff r(N)$ such that
$\frac{d}{dt}D^np_t$ is piecewise continuous in $t$, for each $0\leq
n\leq r$. Such $p$ will be called a piecewise $C^1$ path in $\Diff
r(N)$. We define the length of $p$ by
\begin{displaymath}
  l_r(p)=\max_{0\leq n\leq r} \max_{y\in N}\int^1_0
  \norm{\frac{d}{dt}D^n_yp_t} dt,
\end{displaymath}
and its partial length is given by
\begin{displaymath}
  l_r(p;s):= \max_{0\leq n\leq r} \max_{y\in N}
  \int^s_0\norm{\frac{d}{dt}D^n_yp_t} dt. 
\end{displaymath}

This length structure on $\Diff r(N)$ allows us to define a metric
$d_r$ as follows: given $h,g\in\Diff r(N)$ we define
\begin{equation}
  \label{eq:dr-distance-def}
  d_r(h,g):= \inf_{p\in \mathcal{P}}\max \left\{l_r(p),
    l_r(p^{-1})\right\} 
\end{equation}
where
\begin{displaymath}
  \mathcal{P}:=\left\{ p\in C^1_{\mathrm{pw}}\big([0,1], \Diff
    r(N)\big) : p(0)=h,\ p(1)=g\right\}
\end{displaymath}

As it was already mentioned above, this metric $d_r$ induces the usual
$C^r$-uniform topology on $\Diff{r}(N)$ which coincides with the one
induced by the Banach manifold structure.

Finishing this subsection, we recall four lemmas from \cite[Lemmas
5.1, 5.2, 5.3 and 5.5]{delaLlaveWindsorLivThmNonComm} that we will use
in forthcoming sections:

\begin{lemma}
  \label{lem:norm-Cr-vs-length}
  For every $C^1$-path $p\colon\R\to\Diff r(N)$, every $1\leq k\leq
  r$, and all $s>0$, it holds
  \begin{displaymath}
    \norm{D^kp_s} \leq e^{\kappa
      l_0(p;s)}\left(l_k(p;s)+\norm{D^kp_0}\right).
  \end{displaymath}
  where $\norm{\cdot}$ denotes the operator norm for linear maps from
  $(TN)^{\bigotimes k}$ to $TN$, and $\kappa$ is a real constant
  depending on the Riemannian metric of $N$. In particular, it holds
  \begin{displaymath}
    \norm{Dp_s}_{r-1}\leq e^{\kappa l_0(p;s)}
    \left(l_r(p;s)+\norm{Dp_0}_{r-1}\right). 
  \end{displaymath}
\end{lemma}

\begin{lemma}
  \label{lem:l-r-1-pre-compos-path}
  Let $p\in C^1\big([0,1],\Diff{r-1}(N)\big)$ and $h\in\Diff
  r(N)$. Then, there exists a constan $C>0$ depending just on $r$ such
  that
  \begin{displaymath}
    l_{r-1}\big(h\circ p_s\big) \leq C\norm{Dh}_{r-1}
    \left(1+\max_{s\in [0,1]}\norm{Dp_s}_{r-2}\right)^{r-1}
    l_{r-1}(p_s),\quad \forall s\in [0,1].
  \end{displaymath}
  On the other hand, if $p\in C^1\big([0,1],\Diff{r}(N)\big)$ and
  $h\in\Diff r(N)$, then
  \begin{displaymath}
    l_{r}(p_s \circ h)\leq C\max_{k_1,\ldots ,k_r}\norm{D^1h}^{k_1}
    \cdots \norm{D^rh}^{k_r}l_r(p_s), \quad\forall s\in [0,1],
  \end{displaymath}
  where the maximum is taken over all $k_1,\ldots ,k_r\geq 0$ such
  that
  \begin{displaymath}
    k_1+2k_2+\ldots +rk_r\leq r.
  \end{displaymath}
  Crudely, this may be estimated by
  \begin{displaymath}
    l_{r}(p_s \circ h)\leq C
    \left(1+\norm{Dh}_{r-1}\right)^{r}l_r(p_s), \quad\forall s\in [0,1].
  \end{displaymath}
\end{lemma}

\begin{lemma}
  \label{lem:dr-1-distance-pre-compos}
  Let $C>0$ and $r\in\N$ be arbitrary. Suppose $h\in\Diff r(N)$ and
  $g_1,g_2\in\Diff{r-1}(N)$. Then, there exists a constant $C'=C'(C,r,N)>0$
  depending only on $C$, $r$ and the manifold $N$, such that
  \begin{align*}
    d_{r-1}\left(h\circ g_1,h\circ g_2\right)& < C'd_{r-1}(g_1,g_2), \\
    d_{r-1}\left(g_1\circ h,g_2\circ h\right)&<C'd_{r-1}(g_1,g_2),
  \end{align*}
  whenever
  \begin{displaymath}
    d_r(h,Id)<C,\ d_{r-1}(g_1,Id)<C,\ \text{and } d_{r-1}(g_2,Id)<C.
  \end{displaymath}
\end{lemma}

\begin{lemma}
  \label{lem:dr-distance-cocycles}
  Given a Lipschitz continuous cocycle $\alpha\colon\Z\times M\to\Diff
  r(N)$ over a homeomorphism $f\colon M\carr$, let us define define
  \begin{displaymath}
    \rho _0:=\max_{w\in M}d_0\big(\alpha^{(1)}(w),Id_N\big),
  \end{displaymath}
  and
  \begin{displaymath}
    \rho_1:=\max_{w\in M} \max\left\{\norm{D\alpha^{(1)}(w)},
      \norm{D\alpha^{(1)}(w)^{-1}}\right\}.
  \end{displaymath}
  Then we have the following estimates for $m\leq r$
  \begin{displaymath}
    d_0\big(\alpha ^{(n)}(w),Id_N\big)\leq \rho _0\abs{n},\ \text{and
    }\norm{D^m\alpha ^{(n)}(w)}\leq C\rho_1^{m\abs{n}}, 
  \end{displaymath}
  for all $(n,w)\in \Z\times M$.
\end{lemma}

\subsection{Hölder cocycles and domination}
\label{sec:hold-cocycl-domination}

Let us consider a hyperbolic homeomorphism $f\colon M\carr$ on a
compact metric space $(M,d)$ and $\alpha\colon\Z\times
M\to\Diff{r}(N)$ be a cocycle over $f$, where $N$ denotes a smooth
compact manifold. Given $\nu>0$, we say that $\alpha$ is a
\emph{$\nu$-Hölder cocycle} when there exists a constant $C_2>0$ such
that
\begin{equation}
  \label{eq:Holder-condition-alpha}
  d_r\Big(\alpha^{(1)}(x),\alpha^{(1)}(y)\Big)\leq C_2 d(x,y)^{\nu},
  \quad\forall x,y\in M.  
\end{equation}

On the other hand, let us define $\rho=\rho(\alpha)>0$ by
\begin{equation}
  \label{eq:rho-def}
  \rho:=\max_{y\in N} \max _{x\in
    M}\left\{\norm{D\big(\alpha^{(1)}(x)\big)_y},
    \norm{D\big(\alpha^{(1)}(x)^{-1}\big)_y }\right\},
\end{equation}
and let $\lambda>0$ be the constant associated to $f$ given by
Definition~\ref{def:hyperbolic-homeo}. Then, given a real number
$t>0$, we will say that the $\nu$-Hölder cocycle $\alpha$ is
\emph{$t$-dominated} whenever
\begin{displaymath}
  \rho^t\lambda ^{\nu}<1.
\end{displaymath}

In what follows, for simplicity of the presentation we will assume
$\nu = 1$; the general case is entirely analogous.

\section{Constructing Invariant Holonomies}
\label{sec:inv-holonomies}

In this section we start the proof of Theorem~\ref{thm:main-1}. The
first step of the proof is
Proposition~\ref{pro:holonomies-existence-regularity}, where we show
the existence of invariant holonomies for $(2r-1)$-dominated
cocycles. This result is mainly inspired by Proposition 2.5 in
\cite{VianaAlmostAllCocyc} and Theorem 5.8 in
\cite{delaLlaveWindsorLivThmNonComm}.

\begin{proposition}
  \label{pro:holonomies-existence-regularity}
  Let $f\colon M\carr$ be a hyperbolic homeomorphism on a compact
  metric space $(M,d)$ and $\alpha\colon M\to\Diff r(N)$ be a
  $(2r-1)$-dominated $1$-Hölder cocycles over $f$. Then there exists a
  constant $C_4>0$ such that, for any $x\in M$ and any $y, z\in
  W^s(x,f)$ the limit
  \begin{displaymath}
    H^{s,\alpha}_{yz}:= \lim_{n\to+\infty}
    \alpha^{(n)}(z)^{-1}\circ\alpha^{(n)}(y)
  \end{displaymath}
  exists in the metric space $\big(\Diff{r-1}(N),d_{r-1}\big)$ and
  \begin{displaymath}
    d_{r-1}(H^{s,\alpha}_{yz}, Id_N) \leq C_4 d(y,z),
  \end{displaymath}
  whenever $y,z\in W^s_\varepsilon(x,f)$, where the constant
  $\varepsilon>0$ associated to $f$ is given by
  Definition~\ref{def:hyperbolic-homeo}.

  On the other hand, if $y, z\in W^u(x,f)$, we can anogously define
  \begin{displaymath}
    H^{u,\alpha}_{yz}:=\lim _{n\rightarrow \infty}
    \alpha^{(-n)}(z)^{-1} \circ \alpha^{(-n)}(y), 
  \end{displaymath}
  and the very same Hölder estimtes holds for these maps when $y,z\in
  W^u_\varepsilon(x,f)$.

  Finally, for every $x\in M$ and $\sigma\in\{s,u\}$, it holds
  \begin{displaymath}
    H^{\sigma,\alpha}_{yz}=H^{\sigma,\alpha}_{xz}\circ
    H^{\sigma,\alpha}_{yx},
  \end{displaymath}
  and
  \begin{displaymath}
    H^{\sigma,\alpha}_{f^n(y)f^n(z)}\circ \alpha^{(n)}(y)=
    \alpha^{(n)}(z)\circ H^{\sigma,\alpha}_{yz}, 
  \end{displaymath}
  for every $y,z\in W^\sigma(x,f)$ and every $n\in\Z$.
\end{proposition}

\begin{definition}
  \label{def:su-holonomies}
  The maps $H^{s,\alpha}$ and $H^{u,\alpha}$ given by
  Proposition~\ref{pro:holonomies-existence-regularity} are called
  \emph{stable} and \emph{unstable holonomies}, respectively.
\end{definition}

\begin{remark}
  Assuming that $\alpha$ is just $1$-dominated in
  Proposition~\ref{pro:holonomies-existence-regularity}, it can be
  rather easily shown that the invariant holonomies exists, but in
  general they are only $C^0$.
\end{remark}

\begin{proof}[Proof of
  Proposition~\ref{pro:holonomies-existence-regularity}]
  After taking forward iterates if necessary, we can assume that
  $y,z\in W^s_\varepsilon(x)$. Let us define
  $\gamma_n(y,z):=\alpha^{(n)}(z)^{-1}\circ\alpha^{(n)}(y)$. Then let
  us estimate the $d_{r-1}$-distance between $\gamma_{n+1}(y,z)$ and
  $\gamma_n(y,z)$. First of all observe that
  \begin{displaymath}
    \gamma _{n+1}(y,z)=\alpha^{(n)}(z)^{-1}\circ
    \alpha^{(1)}\big(f^n(z)\big)^{-1}\circ 
    \alpha^{(1)}\big(f^n(y)\big) \circ \alpha^{(n)}(y).
  \end{displaymath}
  
  Then, let us consider a continuous path $p\colon
  [0,1]\to\Diff{r}(N)$ connecting the diffeomorphism
  $\alpha^{(1)}(f^n(z))^{-1} \circ\alpha^{(1)}(f^n(y))$ to
  $Id_N$. Since, by compactness of $M$,
  $d_{r-1}\big(\alpha^{(1)}(w_1)^{-1}\circ \alpha^{(1)}(w_2),
  Id_N\big)$ is uniformly bounded for any $w_1,w_2\in M$, we may
  assume that the $d_{r-1}$-lengths $l_{r-1}(p_s)$ and
  $l_{r-1}(p_s^{-1})$ are also uniformly bounded, independently of $x,
  y, z$ and $n$.

  In what follows we will use the letter $C$ to denote a positive
  constant that may differ in each step. By
  Lemma~\ref{lem:l-r-1-pre-compos-path}, we know that
  \begin{displaymath}
    l_{r-1}\big(\alpha^{(n)}(z)^{-1}\circ p_s\big) \leq
    C\norm{D\alpha^{(n)}(z)^{-1}}_{r-1} \Big(1+\max_{s\in [0,1]}
    \norm{Dp_s}_{r-2}\Big)^{r-1}l_{r-1}(p_s), 
  \end{displaymath}
  for every $s\in [0,1]$, and by Lemma~\ref{lem:norm-Cr-vs-length} and
  Lemma~\ref{lem:dr-distance-cocycles} we know that the right hand is
  less or equal than $C\rho ^{rn}l_{r-1}(p_s)$, where $C$ is
  independent of $n$. Now, by Lemma~\ref{lem:l-r-1-pre-compos-path} it
  easily follows that
  \begin{displaymath}
    \begin{split}
      &l_{r-1}\big(\alpha^{(n)}(z)^{-1}\circ p_s
      \circ\alpha^{(n)}(y)\big)\\
      &\leq C \max_{k_1,\ldots ,k_{r-1}}
      \norm{D^1\alpha^{(n)}(y)}^{k_1}
      \cdots\norm{D^{r-1}\alpha^{(n)}(y)}^{k_{r-1}}
      l_{r-1}\big(\alpha^{(n)}(z)^{-1} \circ p_s\big) \\
      & \leq C\rho^{rn}\max_{k_1,\ldots ,k_{r-1}}
      \norm{D^1\alpha^{(n)}(y)}^{k_1}
      \cdots\norm{D^{r-1}\alpha^{(n)}(y)}^{k_{r-1}}l_{r-1}(p_s),
    \end{split}
  \end{displaymath}
  for every $s\in [0,1]$.

  Then, applying Lemma~\ref{lem:dr-distance-cocycles}, we get that
  \begin{displaymath}
    l_{r-1}\big(\alpha^{(n)}(z)^{-1}\circ p_s
    \circ\alpha^{(n)}(y)\big)\leq C \rho^{rn}\rho^{(r-1)n}l_{r-1}(p_s).  
  \end{displaymath}
  
  On the other hand, by symmetry the very same estimate holds for the
  inverse and thus,
  \begin{displaymath}
    d_{r-1}\big(\gamma _{n+1}(y,z), \gamma _n(y,z)\big)\leq
    C\rho^{(2r-1)n} d_{r-1}\Big(\alpha^{(1)}\big(f^n(z)\big)^{-1}
    \circ \alpha^{(1)}\big(f^n(y)\big), Id_N\Big).
  \end{displaymath}

  Now, as $M$ is compact, there is a constant $C>0$ such that
  $d_r\big(\alpha^{(1)}(w),Id_N\big)\leq C$, for all $w\in M$ and
  consequently, by Lemma~\ref{lem:dr-1-distance-pre-compos}, there
  exists $K>1$ such that
  \begin{displaymath}
    d_{r-1}\big(\alpha^{(1)}(w_1)^{-1}\circ\alpha^{(1)}(w_2),
    Id_N\big)\leq K d_{r-1}\big(\alpha^{(1)}(w_1),
    \alpha^{(1)}(w_2)\big),
  \end{displaymath}
  for every $w_1,w_2 \in M$. Thus,
  \begin{displaymath}
    \begin{split}
      d_{r-1}\big(\gamma_{n+1}(y,z), &\gamma_n(y,z)\big) \\
      & \leq C\rho^{(2r-1)n} d_{r-1}\big(\alpha^{(1)}(f^n(z))^{-1}
      \circ \alpha^{(1)}(f^n(y)), Id_N\big) \\
      & \leq CK\rho^{(2r-1)n} d_{r-1}\big(\alpha^{(1)}(f^n(z)),
      \alpha^{(1)}(f^n(y))\big),
    \end{split}
  \end{displaymath}
  where the last line of the estimate, by the Hölder condition on
  $\alpha$, is less or equal than
  $CKC_2\rho^{(2r-1)n}d(f^n(z),f^n(y))$, where $C_2$ is the positive
  constant given by the Hölder condition
  \eqref{eq:Holder-condition-alpha}. Then, invoking our assumption
  that $y,z\in W^s_{\varepsilon}(x)$, we conclude that
  \begin{displaymath}
    \begin{split}
      d_{r-1}\big(\gamma _{n+1}(y,z), \gamma _n(y,z)\big) & \leq
      CKC_2\rho^{(2r-1)n} d(f^n(z),f^n(y))\\
      & \leq CKC_2\rho ^{(2r-1)n}\lambda^n d(z,y).
    \end{split}
  \end{displaymath}

  Since $\rho ^{2r-1}\lambda <1$, this proves that the sequence is
  Cauchy in the metric space $\big(\Diff{r}(N),d_{r-1}\big)$ and
  hence, the limit $H^{s,\alpha}_{y,z}$ exists and satisfies
  \begin{displaymath}
    d_{r-1}(H^{s,\alpha}_{y,z}, Id)\leq C_4 d(y,z), 
  \end{displaymath}
  whith $C_4:=\sum_{n=0}^{\infty}CKC_2(\rho ^{(2r-1)}\lambda )^n$
\end{proof}

\section{Constructing a Transfer Map Assuming $\Fix (f)\neq \emptyset$}
\label{sec:constr-trans-map}

In this section we continue with the proof of
Theorem~\ref{thm:main-1}, using the invariant holonomies constructed
in \S\ref{sec:inv-holonomies} to define the transfer map $P$ that
solves the cohomological equation \eqref{eq:cohomo-eq-main-thm} in the
case $f$ exhibits a fixed point.

So, we will assume there exists $x\in M$ such that $f(x)=x$. For such
a point, we write $W(x):=W^s(x)\cap W^u(x)$. We start defining
$P\colon W(x)\to\Diff{r-1}(N)$ by
\begin{displaymath}
  P(y)=H^{s,\alpha}_{xy}\circ (H^{s,\beta}_{xy})^{-1}=
  H^{s,\alpha}_{xy}\circ H^{s,\beta}_{yx},
\end{displaymath}
where $H^{\sigma,\alpha}$ and $H^{\sigma,\beta}$ are the holonomy maps
given by Proposition~\ref{pro:holonomies-existence-regularity}.

Note that $P$ satisfies the following relation:
\begin{equation}
  \label{eq:cohom-eq-on-W}
  \alpha^{(n)}(y)=P(f^n(y))\circ \beta^{(n)}(y)\circ P(y)^{-1},
  \quad\forall y\in W(x),\ \forall n\in\N.
\end{equation}

Indeed,
\begin{displaymath}
  \begin{split}
    P(f(y))&=H^{s,\alpha}_{xf(y)}\circ
    H^{s,\beta}_{f(y)x}=H^{s,\alpha}_{f(x)f(y)}\circ
    H^{s,\beta}_{f(y)f(x)} \\
    & =\alpha^{(1)}(y)\circ H^{s,\alpha}_{xy}\circ
    \alpha^{(1)}(x)^{-1}\circ \beta^{(1)}(x)\circ
    H^{s,\beta}_{yx}\circ \beta^{(1)}(y)^{-1} \\
    &=\alpha^{(1)}(y)\circ H^{s,\alpha}_{xy}\circ
    H^{s,\beta}_{yx}\circ\beta^{(1)}(y)^{-1}\\
    &= \alpha^{(1)}(y)\circ P(y)\circ \beta^{(1)}(y)^{-1}.
  \end{split}
\end{displaymath}

Combining this equation with the cocycle property
\eqref{eq:cocycle-prop}, we can easily get
\eqref{eq:cohom-eq-on-W}. Observe we have strongly used the existence
of a fixed point, the periodic orbit condition \eqref{eq:POC} and
Proposition~\ref{pro:holonomies-existence-regularity}.

Up to now we have our transfer map $P$ defined on a dense subset of
$M$ where it satisfies the desired property \eqref{eq:cohom-eq-on-W}.
Then, we will show that $P$ can be extended to the whole space $M$. In
order to do this, we will prove that $P$ is Lipschitz on $W(x)$ and
the following lemma plays a key role in proving that:
\begin{lemma}
  \label{lem:transfer-P-s-u-holon}
  If $\alpha$ and $\beta$ are $1$-dominated, then
  \begin{displaymath}
    P(y)=H^{s,\alpha}_{xy}\circ H^{s,\beta}_{yx}=
    H^{u,\alpha}_{xy}\circ H^{u,\beta}_{yx},\quad \forall y\in W(x).
  \end{displaymath}
\end{lemma}

The following classical result (see for instance \cite[Corollary
6.4.17]{KatokHasselblatt}) will be used in the proof:

\begin{lemma}[Anosov Closing Lemma]
  \label{lem:Anosov-closing-lemma}
  Let $f$ and $\lambda>0$ be as in
  Definition~\ref{def:hyperbolic-homeo}. Then, given any $\theta
  >\lambda$, there exist $C_5>0$ and $\varepsilon _0>0$ such that for
  any $z\in M$ and every $n\in\Z$ satisfying
  $d(f^n(z),z)<\varepsilon_0$, there exists a periodic point
  $p\in\Fix(f^n)$ such that
  \begin{displaymath}
    d(f^j(z),f^j(p))\leq C_5 \theta^{\min\{j,n-j\}} d(f^n(z),z),
    \quad\text{for } j=0,1,\ldots, n.
  \end{displaymath}
\end{lemma}

\begin{proof}[Proof of Lemma~\ref{lem:transfer-P-s-u-holon}]
  Let us consider the distance function $\tilde d$ on $\Diff{r}(N)$
  given by
  \begin{displaymath}
    \tilde{d}(g,h)=\sup_{y\in N}d(g(y),h(y)), \quad\forall
    g,h\in\Diff{r}(N). 
  \end{displaymath}
  Observe that the distance $\tilde d$ exhibits the following
  properties:
  \begin{itemize}
  \item it is right invariant, \ie $\tilde{d}(g\circ u,h\circ
    u)=\tilde{d}(g,h)$ for every $g,h,u\in\Diff{r}(N)$;
  \item $\tilde{d}(u\circ g,u\circ h)\leq(\Lip u)\tilde{d}(g,h)$ for
    any $g,h,u\in\Diff{r}(N)$ and where $\Lip u$ denotes Lipschitz
    constant of $u$ given by $\Lip u:=\sup_{x\neq
      y}d(u(x),u(y))/d(x,y)$. Notice that $\Lip u\leq \norm{Du}$.
  \end{itemize}

  Let $\lambda>0$ be the hyperbolic constant associated to $f$ given
  by Definition~\ref{def:hyperbolic-homeo}. Let us fix
  $\theta\in(\lambda,1)$ such that $\rho\theta<1$, where
  $\rho(\alpha)$ and $\rho(\beta)$ are given by \eqref{eq:rho-def} and
  $\rho:=\max\{\rho(\alpha),\rho(\beta)\}>0$. Let $C_5>0$ and
  $\varepsilon_0 >0$ be the constants given by
  Lemma~\ref{lem:Anosov-closing-lemma}.

  Fix an arbitrary point $y\in W(x)$. We begin by noticing that, since
  $y\in W(x)$, there exist $C_6>0$ and $n_0\in\N$ such that it holds
  \begin{displaymath}
    d\big(f^{-n}(y),f^n(y)\big)\leq C_6\lambda ^{n-n_0}, \quad\forall
    n\geq n_0. 
  \end{displaymath}
  In fact, this easily follows from the fact that, as $y\in
  W(x)=W^s(x)\cap W^u(x)$, there exists $n_0\in \mathbb{N}$ such that
  $f^{n_0}(y)\in W^s_{\varepsilon }(x)$ and $f^{-n_0}(y)\in
  W^u_{\varepsilon }(x)$ and the exponential convergence towards $x$
  in $W^s_{\varepsilon}(x)$ and $W^u_{\varepsilon}(x)$.
  
  Now, let $n_1\geq n_0$ be such that for all $n\geq n_1$ it holds
  $d(f^n(y), f^{-n}(y))<\varepsilon_0$. Thus, by
  Lemma~\ref{lem:Anosov-closing-lemma}, for every $n\geq n_1$ there
  exists a periodic point $p_n\in\Fix(f^{2n})$ such that
  \begin{displaymath}
    d\Big(f^j\big(f^{-n}(p_n)\big), f^j\big(f^{-n}(y)\big)\Big) \leq
    C_5\theta^{\min\{j, 2n-j\}} d\big(f^{-n}(y),f^n(y)\big),
  \end{displaymath}
  for each $j=0,1,\ldots ,2n$. Using the periodic orbit condition
  \eqref{eq:POC} and noticing that $f^{2n}(f^{-n}(p_n))=f^{-n}(p_n)$,
  we get
  \begin{displaymath}
    \alpha^{(2n)}\big(f^{-n}(p_n)\big)=
    \beta^{(2n)}\big(f^{-n}(p_n)\big),
  \end{displaymath}
  which can be rewritten as
  \begin{displaymath}
    \alpha^{(n)}(p_n)\circ\alpha^{(n)}\big(f^{-n}(p_n)\big)
    =\beta^{(n)}(p_n) \circ \beta^{(n)}\big(f^{-n}(p_n)\big), 
  \end{displaymath}
  or equivalently,
  \begin{displaymath}
    \alpha^{(n)}(f^{-n}(p_n))\circ
    \beta^{(n)}\big(f^{-n}(p_n)\big)^{-1} =
    \alpha^{(n)}(p_n)^{-1}\circ \beta^{(n)}(p_n).
  \end{displaymath}

  Hence, since
  \begin{displaymath}
    \begin{split}
      \alpha^{(n)}&\big(f^{-n}(p_n)\big) =
      \alpha^{(1)}\big(f^{-1}(p_n)\big)\circ\cdots\circ
      \alpha^{(1)}\big(f^{-n}(p_n)\big) \\
      & = \Big(\alpha^{(1)}\big(f^{-n}(p_n)\big)^{-1} \circ\cdots\circ
      \alpha^{(1)}\big(f^{-1}(p_n)\big)^{-1}\Big)^{-1} =
      \alpha^{(-n)}(p_n)^{-1}
    \end{split}
  \end{displaymath}
  and analogously,
  \begin{displaymath}
    \beta^{(n)}\big(f^{-n}(p_n)\big)^{-1}=\beta^{(-n)}(p_n),
  \end{displaymath}
  it follows that
  \begin{equation}\label{equa: a}
    \alpha^{(-n)}(p_n)^{-1}\circ \beta^{(-n)}(p_n) =
    \alpha^{(n)}(p_n)^{-1} \circ\beta^{(n)}(p_n). 
  \end{equation}
  Now, we claim that
  \begin{equation}
    \label{eq:P-def-on-y-s}
    \tilde{d}\Big(\alpha^{(n)}(y)^{-1} \circ\beta^{(n)}(y),
    \alpha^{(n)}(p_n)^{-1} \circ\beta^{(n)}(p_n)\Big)
    \rightarrow 0,
  \end{equation}
  and
  \begin{equation}
    \label{eq:P-def-on-y-u}
    \tilde{d}\Big(\alpha^{(-n)}(y)^{-1} \circ \beta^{(-n)}(y),
    \alpha^{(-n)}(p_n)^{-1}\circ \beta^{(-n)}(p_n)\Big)\rightarrow 0,
  \end{equation}
  when $n\to+\infty$.

  Thus, it follows from (\ref{equa: a}) and our claim that
  \begin{displaymath}
    \tilde{d}\big(\alpha ^{(n)}(y)^{-1}\circ \beta ^{(n)}(y),\alpha
    ^{(-n)}(y)^{-1}\circ \beta ^{(-n)}(y)\big)\to 0,
  \end{displaymath}
  when $n\to+\infty$. Observing that
  \begin{displaymath}
    \begin{split}
      &\tilde{d}\big(\alpha ^{(n)}(y)^{-1}\circ \beta ^{(n)}(y),
      H^{s,\alpha}_{xy}\circ H^{s,\beta}_{yx}\big) \\
      & = \tilde{d}\big(\alpha^{(n)}(y)^{-1}\circ\alpha^{(n)}(x)\circ
      \beta ^{(n)}(x)^{-1} \circ \beta^{(n)}(y),H^{s,\alpha}_{xy}\circ
      H^{s,\beta}_{yx}\big)\rightarrow 0,
    \end{split}
  \end{displaymath}
  as $n\to+\infty$, and
  \begin{displaymath}
    \begin{split}
      &\tilde{d}\big(\alpha^{(-n)}(y)^{-1}\circ \beta^{(-n)}(y),
      H^{u,\alpha}_{xy}\circ H^{u,\beta}_{yx}\big) \\
      &= \tilde{d}\big(\alpha^{(-n)}(y)^{-1}\circ
      \alpha^{(-n)}(x)\circ \beta^{(-n)}(x)^{-1}
      \beta^{(-n)}(y),H^{u,\alpha}_{xy}\circ
      H^{u,\beta}_{yx}\big)\rightarrow 0,
    \end{split}
  \end{displaymath}
  as $n\to+\infty$, we would get
  \begin{displaymath}
    P(y)=H^{s,\alpha}_{xy}\circ H^{s,\beta}_{yx}=H^{u,\alpha}_{xy}\circ
    H^{u,\beta}_{yx},
  \end{displaymath}
  as desired.

  So, in order to complete the proof it remains to prove our claim. In
  fact, we shall only show \eqref{eq:P-def-on-y-s}. The other case
  \eqref{eq:P-def-on-y-u} is completely analogous.

  As before, in what follows we will use $C$ as a generic notation for
  positive constants that may differ at each step.

  We begin by noticing that
  \begin{equation}
    \label{eq:alpha-n-y-alpha-n-pn-estimate}
    \begin{split}
      &\tilde{d}\big(\alpha^{(n)}(y)\circ\alpha^{(n)}(p_n)^{-1},
      Id\big)\\
      &\leq\sum_{j=0}^{n-1}\tilde{d}\Big(\alpha^{(n-j)}(f^j(y))\circ
      \alpha^{(n-j)}(f^j(p_n))^{-1}, \\
      &\qquad\qquad\qquad\qquad\qquad\qquad
      \alpha^{(n-j-1)}(f^{j+1}(y))\circ
      \alpha^{(n-j-1)}(f^{j+1}(p_n))^{-1}\Big) \\
      &= \sum_{j=0}^{n-1}\tilde{d}\Big(\alpha^{(n-j)}(f^{j}(y))
      \circ\alpha^{(1)}(f^j(p_n))^{-1}\circ
      \alpha^{(n-j-1)}(f^{j+1}(p_n))^{-1}, \\
      & \qquad\qquad\qquad\qquad\qquad\qquad\alpha
      ^{(n-j-1)}(f^{j+1}(y))\circ
      \alpha^{(n-j-1)}(f^{j+1}(p_n))^{-1}\Big)\\
      &= \sum_{j=0}^{n-1}\tilde{d}\Big(\alpha^{(n-j)}(f^{j}(y))
      \circ\alpha^{(1)}(f^j(p_n))^{-1},
      \alpha^{(n-j-1)}(f^{j+1}(y))\Big)\\
      &\leq\sum_{j=0}^{n-1}
      \Lip\Big(\alpha^{(n-j-1)}\big(f^{j+1}(y)\big)\Big)
      \tilde{d}\Big(\alpha^{(1)}\big(f^j(y)\big)\circ
      \alpha^{(1)}\big(f^j(p_n)\big)^{-1}, Id\Big).
    \end{split}
  \end{equation}

  On the other hand, observe that
  \begin{equation}
    \label{eq:alpha-1-alpha-1-pn-Lip-estimate}
    \begin{split}
      &\tilde{d}\Big(\alpha^{(1)}\big(f^j(y)\big)\circ
      \alpha^{(1)}\big(f^j(p_n)\big)^{-1}, Id\Big) = \tilde{d}\Big(
      \alpha^{(1)}\big(f^j(y)\big),\alpha^{(1)}\big(f^j(p_n)\big)\Big)\\
      & \leq C d\big(f^j(y),f^j(p_n)\big)= C
      d\Big(f^{n+j}\big(f^{-n}(y)\big),
      f^{n+j}\big(f^{-n}(p_n)\big)\Big) \\
      & \leq C\theta^{\min\{n+j,2n-(n+j)\}} d(f^{-n}(y),f^n(y))=C
      \theta^{n-j}d(f^{-n}(y),f^n(y)),
    \end{split}
  \end{equation}
  for $j\in \lbrace 0,1,\ldots ,n-1\rbrace$. Thus, combining
  \eqref{eq:alpha-n-y-alpha-n-pn-estimate} and
  \eqref{eq:alpha-1-alpha-1-pn-Lip-estimate}, and recalling that
  $\norm{D_N\alpha^{(k)}(z)}\leq\rho^k$ for all $z\in M$,
  $\rho\theta <1$ and $d(f^{-n}(y),f^n(y))\leq C_6 \lambda ^{n-n_0}$,
  we get
  \begin{displaymath}
    \begin{split}
      &
      \tilde{d}\big(\alpha^{(n)}(y)\circ\alpha^{(n)}(p_n)^{-1},Id\big)
      \\
      & \leq\sum_{j=0}^{n-1}\norm{D\alpha^{(n-j-1)}(f^{j+1}(y))}
      C\theta^{n-j}d\big(f^{-n}(y),f^n(y)\big) \\
      &\leq C\sum_{j=0}^{n-1} \rho^{n-j}\theta^{n-j}
      d\big(f^{-n}(y),f^n(y)\big) \leq C
      d(f^{-n}(y),f^n(y))\sum_{j=0}^{\infty}
      (\rho\cdot\theta)^{n-j} \\
      &\leq C d(f^{-n}(y),f^n(y))\leq C \lambda^{n-n_0}.
    \end{split}
  \end{displaymath}

  Analogously we can prove
  \begin{displaymath}
    \tilde{d}\big(\beta^{(n)}(p_n)\circ\beta^{(n)}(y)^{-1},Id\big) \leq
    C\lambda^{n-n_0}.  
  \end{displaymath}

  Now observe that, since $0<\rho\lambda<1$, it follows
  \begin{equation}
    \label{eq:b}
    \begin{split}
      &\tilde{d}\big(\alpha^{(n)}(y)^{-1}\circ\beta^{(n)}(y),
      \alpha^{(n)}(p_n)^{-1}\circ\beta^{(n)}(p_n)\big) \\
      &\leq\Lip(\alpha^{(n)}(p_n))\tilde{d}\big(\alpha^{(n)}(p_n)\circ
      \alpha^{(n)}(y)^{-1}\circ\beta^{(n)}(y),\beta^{(n)}(p_n)\big) \\
      & =\Lip(\alpha^{(n)}(p_n))\tilde{d}\big(
      \alpha^{(n)}(p_n)\circ\alpha^{(n)}(y)^{-1},
      \beta^{(n)}(p_n)\circ\beta^{(n)}(y)^{-1}\big) \\
      &\leq\norm{D\alpha^{(n)}(p_n)}\left(\tilde{d}\big(
        \alpha^{(n)}(p_n)\circ\alpha^{(n)}(y)^{-1},Id\big) +
        \tilde{d}\big(Id,\beta^{(n)}(p_n)\circ\beta^{(n)}(y)^{-1}\big)
      \right) \\
      & \leq
      2C\rho^n\lambda^{n-n_0}=2C\lambda^{-n_0}(\rho\lambda)^n\to 0,
    \end{split}
  \end{equation}
  as $n\to+\infty$, and claim \eqref{eq:P-def-on-y-s} is proved.
\end{proof}

Then, we can finally prove that $P$ can be continuously extended to
the whole space $M$:
\begin{lemma}
  \label{lem:P-is-Lipschitz}
  The map $P\colon W(x)\to\Diff{r-1}(N)$ is Lipschitz with respect to
  the $C^{r-4}$-distance.
\end{lemma}

\begin{proof} 
  Let $y$ denote an arbitrary point in $W(x)$ and let us define
  $C_1(y):=d_{r-2}(P(y),Id)+1$.

  Then, let us show there exists a constant $C_2(y)>0$ such that
  \begin{equation}
    \label{eq:2-lemma-P-is-Lipschitz}
    d_{r-2}(H^{s,\alpha}_{xz}\circ H^{s,\beta}_{yx}, Id)<C_2(y),
    \quad\forall z\in W^s_\varepsilon(y), 
  \end{equation} 
  where $C_2(y)$ depends on $\varepsilon$, the constant $C_4$ given by
  Proposition \ref{pro:holonomies-existence-regularity}, $r$, the
  manifold $N$ and $C_1(y)$. Indeed, defining
\begin{displaymath}
  C_2(y):=C'(C_4\varepsilon,r,N)C_1(y) + C_4\varepsilon,
\end{displaymath}  
  where $C'(\cdot,\cdot,\cdot)$ is the constant given by Lemma
  \ref{lem:dr-1-distance-pre-compos}, we get 
  \begin{displaymath}
    \begin{split}
      &d_{r-2}(H^{s,\alpha}_{xz}\circ H^{s,\beta}_{yx}, Id) =
      d_{r-2}(H^{s,\alpha}_{xz}\circ H^{s,\alpha}_{yx}\circ
      H^{s,\alpha}_{xy}\circ H^{s,\beta}_{yx}, Id) \\
      &=d_{r-2}(H^{s,\alpha}_{xz}\circ H^{s,\alpha}_{yx}\circ P(y),
      Id)=d_{r-2}(H^{s,\alpha}_{yz}\circ P(y), Id) \\
      &\leq d_{r-2}(H^{s,\alpha}_{yz}\circ P(y), H^{s,\alpha}_{yz}) +
      d_{r-2}(H^{s,\alpha}_{yz}, Id) \\
      &\leq C'(C_4\varepsilon,r,N)d_{r-2}(P(y),Id) + C_4\varepsilon < C_2(y).
    \end{split}
  \end{displaymath}

 Then, for each $z\in W^s_\varepsilon(y)$ it holds
  \begin{equation}
    \label{eq:1-lemma-P-is-Lipschitz}
    \begin{split}
      &d_{r-3}(P(y), P(z)) = d_{r-3}(H^{s,\alpha}_{xy}\circ
      H^{s,\beta}_{yx}, H^{s,\alpha}_{xz}\circ H^{s,\beta}_{zx}) \\
      &\leq d_{r-3}\big(H^{s,\alpha}_{xy}\circ H^{s,\beta}_{yx},
      H^{s,\alpha}_{xz}\circ H^{s,\beta}_{yx}\big) +
      d_{r-3}\big(H^{s,\alpha}_{xz}\circ H^{s,\beta}_{yx},
      H^{s,\alpha}_{xz}\circ H^{s,\beta}_{zx}\big) \\
      & = d_{r-3}\big(H^{s,\alpha}_{xy}\circ H^{s,\alpha}_{zx} \circ
      H^{s,\alpha}_{xz}\circ H^{s,\beta}_{yx}, H^{s,\alpha}_{xz}\circ
      H^{s,\beta}_{yx}\big) \\
      &\qquad\qquad\qquad\qquad\qquad +
      d_{r-3}\big(H^{s,\alpha}_{xz}\circ H^{s,\beta}_{yx},
      H^{s,\alpha}_{xz}\circ H^{s,\beta}_{yx}\circ
      H^{s,\beta}_{xy}\circ H^{s,\beta}_{zx}\big) \\
      &\leq C'(C_2(y),r,N)\Big(d_{r-3}\big(H^{s,\alpha}_{xy}\circ
      H^{s,\alpha}_{zx},Id\big) + d_{r-3}\big(Id,
      H^{s,\beta}_{xy}\circ H^{s,\beta}_{zx}\big)\Big) \\
      & \leq C_3(y)d(y,z),
    \end{split}
  \end{equation}
  where $C_3(y):=2C_4C'(C_2(y),r,N)$ and the last estimate is
  consequence of Proposition
  \ref{pro:holonomies-existence-regularity}.
  
  Analogously, invoking Lemma \ref{lem:transfer-P-s-u-holon}, we can
  show that for every $y\in W(x)$ it holds
  \begin{equation}
    \label{eq:1-lemma-P-is-Lipschitz-u}
      d_{r-3}(P(y), P(z))\leq C_3(y)d(y,z), \quad\forall z\in W^u_{\varepsilon}(y).
  \end{equation}

  Now, given $y,z\in W(x)$ with $d(y,z)<\tau $, where $\tau$ is the
  constant given by Defintion \ref{def:hyperbolic-homeo}, we can
  consider $w=[y,z]=W^s_{\varepsilon }(y)\cap W^u_{\varepsilon
  }(z)$.
  Then, as $w\in W^s_{\varepsilon }(y)$ and
  $w\in W^u_{\varepsilon }(z)$ by \eqref{eq:1-lemma-P-is-Lipschitz}
  and \eqref{eq:1-lemma-P-is-Lipschitz-u} it follows that
  \begin{displaymath}
    \begin{split}
      &d_{r-3}(P(y), P(z))\leq d_{r-3}(P(y), P(w)) +d_{r-3}(P(w),
      P(z)) \\
      &\leq C_3(y)d(y,w)+ C_3(w)d(w,z).
    \end{split}
  \end{displaymath}
  From the proof, we get that $C_3(w)$ depends on $C_4\varepsilon$,
  $r$, $N$ and $C_1(w)$. Now, note that
  \begin{displaymath}
    d_{r-3}(P(w),Id)\leq d_{r-3}(P(w),P(y))+d_{r-3}(P(y),Id)\leq
    C_3(y)\varepsilon +C_1(y). 
  \end{displaymath}
  In other words, we can estimate $d_{r-3}(P(w),Id)$ from above by some constant
  that depends on $y$ but does not depend on $w$. 
  
  Thus, repeating the previous arguments replacing $r$ by $r-1$, we
  can show that for every $y\in W(x)$ there exists $C_y>0$ such that
  \begin{displaymath}
    d_{r-4}(P(y), P(z))\leq C_yd(y,w)+ C_yd(w,z),
  \end{displaymath}
  for every $z\in W(x)$ such that $d(y,z)<\tau$ and $w=[y,z]$.

  Then, invoking the fact that there exists $D>0$ such
  that 
  \begin{displaymath}
  d(y,w)+ d(w,z)\leq Dd(y,z),
  \end{displaymath}
   for every $y,z$ and $w$ as above, it follows that there exists
  $K_y=K_y(r,N,C_y,f)>0$ such that
  \begin{displaymath}
    d_{r-4}(P(y), P(z))\leq K_yd(y,z),
  \end{displaymath}
  whenever $y,z\in W(x)$ and $d(y,z)<\tau$. Now, as $W(x)$ is dense in
  $M$ and $M$ is compact it follows that there exists a constant $K>0$
  such that, for all $y,z\in W(x)$ with $d(y,z)<\tau$ we have that
  \begin{displaymath}
    d_{r-4}(P(y), P(z))\leq Kd(y,z),
  \end{displaymath}
  i.e. $P$ is Lipschitz with respect to the $C^{r-4}$ topology.
\end{proof}

Therefore, we can extend $P\colon W(x)\to\Diff{r-4}(N)$ to the closure
of $W(x)$ that is the whole space $M$. By continuity, such extension
clearly satisfies \eqref{eq:cohomo-eq-main-thm} and we complete the
proof of Theorem~\ref{thm:main-1} in the case $f$ exhibits a fixed
point.

\section{Constructing a Transfer Map in the General Case}
\label{sec:concluding-the-proof}

In this section we conclude the proof of Theorem~\ref{thm:main-1},
eliminating the assumption of the existence of a fixed point for $f$
that we used in \S\ref{sec:constr-trans-map}.

We shall use the following notation 
\begin{displaymath}
  \begin{split}
    C\big(\alpha ^{(n)},\beta ^{(n)}\big):=\Big\{P\colon
    &M\to\Diff{r-4}(N) : P \text{ is Lipschitz and } \\
    & \alpha^{(n)}(x)=P(f^n(x))\circ\beta^{(n)}(x)\circ P(x)^{-1},\
    \forall x\in M\Big\}.
  \end{split}
\end{displaymath}

To deal with the general case, let $x_0\in M$ be a periodic point and
$n_0$ its period. Consider now the new cocycles
$\tilde{\alpha},\tilde{\beta}: \mathbb{Z}\times M\rightarrow \Diff
{r}(N)$ over $F=f^{n_0}$ given by
\begin{displaymath}
  \tilde{\alpha}^{(1)}(x):=\alpha ^{(n_0)}(x)=
  \alpha^{(1)}\big(f^{n_0-1}(x)\big)\circ \ldots
  \circ\alpha^{(1)}(x),
\end{displaymath}
and
\begin{displaymath}
  \tilde{\beta}^{(1)}(x)=\beta^{(n_0)}(x)=
  \beta^{(1)}\big(f^{n_0-1}(x)\big) \circ\ldots\circ\beta^{(1)}(x). 
\end{displaymath}

It is easy to see that $\tilde{\alpha}$ and $\tilde{\beta}$ are
$(2r-1)$-dominated over $F$, since $f$ is Lipschitz, they are
Lipschitz themselves and $F(x_0)=x_0$. Thus, applying results of
\S\ref{sec:constr-trans-map} to these objects we get a Lipschitz map
$P:M\rightarrow \Diff{r-4}(N)$ satisfying
\begin{displaymath}
  \tilde{\alpha}^{(n)}(y)=P(F^n(y))\circ \tilde{\beta}^{(n)}(y)\circ
  P(y)^{-1},
\end{displaymath}
for all $y\in M$ and $n\in \mathbb{N}$. 

Rewriting this in terms of the original cocycles we get 
\begin{displaymath}
  \alpha ^{(n\cdot n_0)}(y)=P(f^{n\cdot n_0}(y))\circ \beta ^{(n\cdot
    n_0)}(y)\circ P(y)^{-1}, 
\end{displaymath}
for every $y\in M$ and $n\in \mathbb{N}$. In other words,, $P\in
C(\alpha ^{(n_0)},\beta ^{(n_0)}) $ and $P(x_0)=Id$.

What we are going to do in the rest of this section consists in
showing that, in fact, this $P$ is a transfer map for $\alpha$ and
$\beta$, \ie $P\in C(\alpha^{(1)},\beta^{(1)})$.

To prove that, the main step of our argument is the following q
\begin{lemma}
  \label{lem:P-restricted-to-su-manifolds}
  Let $P\in C\big(\alpha ^{(n)},\beta ^{(n)}\big)$ and $x\in M$ be an
  arbitrary point. Then, it holds 
  \begin{displaymath}
    P(z)=H^{s,\alpha}_{xz}\circ
    P(x)\circ H^{s,\beta}_{zx},\quad\forall z\in W^s(x). 
  \end{displaymath}
  An analogous result holds for points on the unstable set $W^u(x)$.
\end{lemma}

\begin{proof} 
  Since $P\in C(\alpha^{(n)},\beta^{(n)})$, we have that
  \begin{displaymath}
    \alpha ^{(n)}(y)=P(f^n(y))\circ \beta ^{(n)}(y)\circ P(y)^{-1},
    \quad\forall y\in M,
  \end{displaymath}
  which can be restated as
  \begin{equation}
    \label{eq:1-lemma-P-restricted-to-su-manifolds}
    \alpha ^{(nk)}(y)=P(f^{nk}(y))\circ \beta ^{(nk)}(y)\circ
    P(y)^{-1},
  \end{equation}
  for all $k\in \mathbb{N}$ and $y\in M$.

  Let $z\in W_{\varepsilon}^s(x)$ be an arbitrary point. Observe
  initially that, by (\ref{eq:1-lemma-P-restricted-to-su-manifolds}),
  we have that
  \begin{displaymath}
    \begin{split}
      &\alpha^{(nk)}(z)^{-1}\circ P(f^{nk}(x))\circ
      \beta^{(nk)}(z)\\
      &=\alpha^{(nk)}(z)^{-1}\circ \alpha^{(nk)}(x)\circ
      P(x)\circ\beta^{(nk)}(x)^{-1} \circ \beta^{(nk)}(z)\to
      H^{s,\alpha}_{xz}\circ P(x)\circ H^{s,\beta}_{zx},
    \end{split}
  \end{displaymath}
  as $k$ goes to infinity.  Now, considering the distance $\tilde{d}$
  defined in the proof of Lemma~\ref{lem:transfer-P-s-u-holon}, we
  have again by (\ref{eq:1-lemma-P-restricted-to-su-manifolds}) that
  \begin{displaymath}
    \begin{split}
      &\tilde{d}(P(z),\alpha ^{(nk)}(z)^{-1}\circ P(f^{nk}(x))\circ
      \beta^{(nk)}(z)) \\
      &= \tilde{d}\big(\alpha^{(nk)}(z)^{-1}\circ P(f^{nk}(z))\circ
      \beta^{(n\cdot k)}(z),\alpha^{(nk)}(z)^{-1}\circ
      P(f^{nk}(x))\circ\beta^{(nk)}(z)\big)\\
      &=\tilde{d}(\alpha ^{(nk)}(z)^{-1}\circ
      P(f^{nk}(z)),\alpha^{(nk)}(z)^{-1}\circ P(f^{nk}(x))) \\
      &\leq\rho^{nk}\tilde{d}\big(P(f^{nk}(z)),P(f^{nk}(x))\big) \leq
      C\rho^{nk}d\big(f^{n\cdot k}(z),f^{n\cdot k}(x)\big) \\
      &\leq C\rho^{nk}\lambda^{nk}d(z,x)=C(\rho\lambda )^{nk}d(z,x)
      \to 0,
    \end{split}
  \end{displaymath}
  when $k$ goes to infinity, because $\rho\lambda <1$.

  Therefore, since 
  \begin{displaymath}
    \alpha ^{(n\cdot k)}(z)^{-1}\circ P(f^{n\cdot
      k}(x))\circ \beta ^{(n\cdot k)}(z)\longrightarrow
    H^{s,\alpha}_{xz}\circ P(x)\circ H^{s,\beta}_{zx},\quad\text{as } k\to\infty,
  \end{displaymath}
  it follows that
  \begin{displaymath}
    P(z)=H^{s,\alpha }_{xz}\circ P(x)\circ H^{s,\beta}_{zx},
  \end{displaymath}
  as desired. The case when $z\in W^s(x)$ follows easily from the
  previous one.
\end{proof}

\begin{remark}
  \label{rem:P-restricted-to-su-manifolds}
  The previous lemma holds with $\alpha$ and $\beta$ being just
  1-dominated. Another superfluous hypothesis is that $P(M)\subset
  \Diff {r-4}(N)$. In fact, the same result holds if $P(M)\subset
  \Diff {s}(N)$ for any $s\geq 1$ and $P:M\rightarrow
  \textrm{Homeo}(N)$ is Lipschitz with respect to the distance
  $\tilde{d}$. This will be used in the last section.
\end{remark}

Let $x\in M$ be a periodic point and assume that its period is
$n_0$. Consider $P\in C(\alpha ^{(n_0)}, \beta ^{(n_0)})$ such that
$P(x)=Id$ which we know that exists by the previous comments, and
$y\in W^s(x)\cap W^u(f^{n_0-1}(x))$ which in particular is such that
$f(y)\in W^u(x)$. By Lemma \ref{lem:P-restricted-to-su-manifolds} we
know that
$$P\mid _{W^s(x)} :W^s(x)\rightarrow \Diff {r-1}(N)$$
is given by $P(z)=H^{s,\alpha}_{xz}\circ H^{s,\beta}_{zx}$ and
$$P\mid _{W^u(x)} :W^u(x)\rightarrow \Diff {r-1}(N)$$
is given by $P(w)=H^{u,\alpha}_{xw}\circ H^{u,\beta}_{wx}$ and
consequently
\begin{displaymath}
  P(y)=H^{u,\alpha}_{xy}\circ H^{u,\beta}_{yx} \; \textrm{and} \;
  P(f(y))=H^{u,\alpha}_{xf(y)}\circ H^{u,\beta}_{f(y)x}. 
\end{displaymath}

We claim now that
$$\alpha ^{(1)} (y)=P(f(y))\circ \beta ^{(1)} (y)\circ P(y)^{-1}$$
that is,
$$P(y)=\alpha ^{(1)} (y)^{-1}\circ P(f(y))\circ \beta ^{(1)}(y).$$
The proof of this claim is similar to the proof of Lemma
\ref{lem:transfer-P-s-u-holon} so we will just indicate how to
proceed.\

Observe initially that
$$P(y)=H^{u,\alpha}_{xy}\circ H^{u,\beta}_{yx}=\lim _{n\rightarrow \infty}\alpha ^{(n\cdot n_0)}(y)^{-1}\circ \alpha ^{(n\cdot n_0)}(x)\circ \beta ^{(n\cdot n_0)}(x)^{-1}\beta ^{(n\cdot n_0)}(y)$$
which by the periodic orbit condition (\ref{eq:POC}) is equal to
$\lim _{n\rightarrow \infty}\alpha ^{(n\cdot n_0)}(y)^{-1} \circ \beta
^{(n\cdot n_0)}(y)$. Analogously
$$P(f(y))=H^{u,\alpha}_{xf(y)}\circ H^{u,\beta}_{f(y)x}=\lim _{n\rightarrow \infty}\alpha ^{(-n\cdot n_0)}(f(y))^{-1}\circ \alpha ^{(-n\cdot n_0)}(x)\circ \beta ^{(-n\cdot n_0)}(x)^{-1}\circ \beta ^{(-n\cdot n_0)}(f(y))$$
$$=\lim _{n\rightarrow \infty}\alpha ^{(-n\cdot n_0)}(f(y))^{-1}\circ \beta ^{(-n\cdot n_0)}(f(y)) $$
and so,
$$\alpha ^{(1)}(y)^{-1}\circ P(f(y))\circ \beta ^{(1)}(y)= \lim _{n\rightarrow \infty}\alpha ^{(-n\cdot n_0+1)}(y)^{-1}\circ \beta ^{(-n\cdot n_0+1)}(y).$$
Thus, what we have to prove is that
$$\lim _{n\rightarrow \infty}\alpha ^{(n\cdot n_0)}(y)^{-1} \circ \beta ^{(n\cdot n_0)}(y) =\lim _{n\rightarrow \infty}\alpha ^{(-n\cdot n_0+1)}(y)^{-1}\circ \beta ^{(-n\cdot n_0+1)}(y).$$

As already mentioned, we will proceed analogously to what we did in
the proof of Lemma \ref{lem:transfer-P-s-u-holon}. Fix $C_5$ and
$\varepsilon _0 >0$ such that the Anosov Closing Lemma
\ref{lem:Anosov-closing-lemma} holds for $\theta \in (\lambda ,1)>0$
such that $\rho . \theta <1$ where $\rho =\rho (\alpha ,\beta)>0$ is
defined in section \ref{sec:hold-cocycl-domination} and consider the
distance $\tilde{d}$ on $\Diff {r}(N)$ as defined in the proof of
Lemma \ref{lem:transfer-P-s-u-holon}. As $y\in W^s(x)$, $f(y)\in
W^u(x)$ and $f^{n_0}(x)=x$, we can find $C_7>0$ and $n_2 \in
\mathbb{N}$ such that for all $n\geq n_2$ we have
$$d(f^{-n\cdot n_0+1}(y),f^{n\cdot n_0}(y))\leq C_7\lambda ^{(n -n_2)\cdot n_0}.$$

Fix $n_3\geq n_2$ such that for all $n\geq n_3$ we have that
$d(f^{-n\cdot n_0+1}(y), f^{n\cdot n_0}(y))<\varepsilon _0$ and thus,
by the Anosov Closing Lemma \ref{lem:Anosov-closing-lemma}, for all
$n\geq n_3$ there exists a periodic point $p_n\in M$ with $f^{2n\cdot
  n_0-1}(p_n)=p_n$ and such that
$$d(f^j(f^{-n\cdot n_0+1}(p_n)), f^j(f^{-n\cdot n_0+1}(y))\leq C_5 \theta ^{\min \lbrace j, 2n\cdot n_0-1-j\rbrace }d(f^{-n\cdot n_0+1}(y),f^{n\cdot n_0}(y))$$
for all $j=0,1,\ldots ,(2n\cdot n_0-1)$. Using the periodic orbit
condition (\ref{eq:POC}) and that $f^{2n\cdot n_0-1}(f^{-n\cdot
  n_0+1}(p_n))=f^{-n\cdot n_0+1}(p_n)$ we get that
$$\alpha ^{(2n\cdot n_0-1)}(f^{-n\cdot n_0+1}(p_n))=\beta ^{(2n\cdot n_0-1)}(f^{-n\cdot n_0+1}(p_n))$$
which can be rewritten as
$$\alpha ^{(n\cdot n_0)}(p_n)\circ \alpha ^{(n\cdot n_0-1)}(f^{-n\cdot n_0+1}(p_n))=\beta ^{(n\cdot n_0)}(p_n)\circ \beta ^{(n\cdot n_0-1)}(f^{-n\cdot n_0+1}(p_n))$$
or equivalently as
$$ \alpha ^{(n\cdot n_0)}(p_n)^{-1}\circ \beta ^{(n\cdot n_0)}(p_n)= \alpha ^{(n\cdot n_0-1)}(f^{-n\cdot n_0+1}(p_n))\circ \beta ^{(n\cdot n_0-1)}(f^{-n\cdot n_0+1}(p_n))^{-1}.$$
Observing now that
$$\alpha ^{(n\cdot n_0-1)}(f^{-n\cdot n_0+1}(p_n))= \alpha ^{(1)} (f^{-1}(p_n))\circ \cdots \circ \alpha ^{(1)}(f^{-n\cdot n_0 +1}(p_n))=\alpha ^{(-n\cdot n_0+1)}(p_n)^{-1} $$
and
$$\beta ^{(n\cdot n_0-1)}(f^{-n\cdot n_0+1}(p_n))^{-1}= \beta ^{(-n\cdot n_0+1)}(p_n)$$
we get that

\begin{equation}\label{eq:d}
  \alpha ^{(n\cdot n_0)}(p_n)^{-1}\circ \beta ^{(n\cdot n_0)}(p_n)=\alpha ^{(-n\cdot n_0+1)}(p_n)^{-1}\circ \beta ^{(-n\cdot n_0+1)}(p_n).
\end{equation}
Following the same lines as in proof of the claim in Lemma
\ref{lem:transfer-P-s-u-holon} we get that both
$$\tilde{d}(\alpha ^{(n\cdot n_0)}(p_n)^{-1}\circ \beta ^{(n\cdot n_0)}(p_n),\alpha ^{(n\cdot n_0)}(y)^{-1} \circ \beta ^{(n\cdot n_0)}(y))$$
and
$$\tilde{d}(\alpha ^{(-n\cdot n_0+1)}(p_n)^{-1}\circ \beta ^{(-n\cdot n_0+1)}(p_n), \alpha ^{(-n\cdot n_0+1)}(y)^{-1}\circ \beta ^{(-n\cdot n_0+1)}(y))$$
goes to zero when $n$ goes to infinity. Thus, using this fact and
(\ref{eq:d}) we get that
$$\lim _{n\rightarrow \infty}\alpha ^{(n\cdot n_0)}(y)^{-1} \circ \beta ^{(n\cdot n_0)}(y) =\lim _{n\rightarrow \infty}\alpha ^{(-n\cdot n_0+1)}(y)^{-1}\circ \beta ^{(-n\cdot n_0+1)}(y)$$
and consequently that
$$\alpha ^{(1)} (y)=P(f(y))\circ \beta ^{(1)} (y)\circ P(y)^{-1}$$
proving our claim.\

Summarizing, we have obtained a map $P\in C(\alpha ^{(n_0)},\beta
^{(n_0)})$ with $P(x)=Id$ where $x\in \Fix (f^{n_0})$ and such that,
for every $y\in W^s(x)\cap W^u(f^{n_0-1}(x))$, it satisfies
\begin{displaymath}
  \alpha ^{(1)} (y)=P(f(y))\circ \beta ^{(1)} (y)\circ P(y)^{-1}.
\end{displaymath} 
Thus, as all the objects involved are uniformly continuous and
$W^s(x)\cap W^u(f^{n_0-1}(x))$ is dense in $M$ it follows that
\begin{displaymath}
  \alpha ^{(1)} (y)=P(f(y))\circ \beta ^{(1)} (y)\circ P(y)^{-1}
\end{displaymath} 
for all $y\in M$, that is, $P\in C(\alpha ^{(1)}, \beta ^{(1)})$ which
completes the proof of the existence part of the Theorem
\ref{thm:main-1}.\

Another way to conclude the proof of the existence part is combining
the fact that there exist $P\in C(\alpha ^{(n_0)},\beta ^{(n_0)})$
such that for some $y\in M$ it satisfies
\begin{displaymath}
  \alpha ^{(1)} (y)=P(f(y))\circ \beta ^{(1)} (y)\circ P(y)^{-1}
\end{displaymath}
with the next lemma which is simple and interesting by itself

\begin{lemma}
  Let $P\in C(\alpha ^{(n_0)}, \beta ^{(n_0)})$. If there exist some
  $y\in M$ such that $\alpha ^{(1)} (y)=P(f(y))\circ \beta
  ^{(1)}(y)\circ P(y)^{-1}$ then $P\in C(\alpha ^{(1)},\beta ^{(1)})$.
\end{lemma}
\begin{proof} By Lemma \ref{lem:P-restricted-to-su-manifolds} we know
  that $$P\mid _{W^s(y)} :W^s(y)\rightarrow \Diff {r-4}(N)$$ is given
  by $P(z)=H^{s,\alpha}_{yz}\circ P(y)\circ H^{s,\beta}_{zy}$. Thus,
$$\alpha ^{(1)}(z)\circ P(z)\circ \beta ^{(1)} (z)^{-1}= \alpha ^{(1)} (z)\circ H^{s,\alpha}_{yz}\circ P(y)\circ H^{s,\beta}_{zy}\circ \beta ^{(1)} (z)^{-1}$$
$$=H^{s,\alpha}_{f(y)f(z)}\circ \alpha ^{(1)} (y)\circ P(y)\circ \beta ^{(1)} (y)^{-1}\circ H^{s,\beta }_{f(z)f(y)}=H^{s,\alpha}_{f(y)f(z)}\circ P(f(y))\circ H^{s,\beta}_{f(z)f(y)}$$
which again by Lemma \ref{lem:P-restricted-to-su-manifolds} is equal
to $P(f(z))$. So,
$$\alpha ^{(1)} (z)=P(f(z))\circ \beta ^{(1)} (z)\circ P(z)^{-1}$$
for all $z\in W^{s}(y)$. Now, as $W^s(y)$ is dense in $M$ and $P$ is
uniformly continuous it follows that
\begin{displaymath}
  \alpha ^{(1)}(z)=P(f(z))\circ \beta ^{(1)}(z)\circ P(z)^{-1}
\end{displaymath}
for all $z\in M$ as we want.
\end{proof}

\begin{remark}
  We would like to stress again that, even though we are working with
  $\nu =1$, all results are valid for any $\nu>0$ and their proofs
  are analogous. Obviously, in the case $\nu \in (0,1)$ our
  constructions above will produce a $\nu$-H\"older map $P:
  M\rightarrow \Diff {r-4}(N)$ instead of a Lipschitz one.
\end{remark}

\section{Improving regularity}
\label{sec:improving-regularity}

The existence part of Theorem \ref{thm:main-1} give us a Lipschitz map
$P: M\rightarrow \Diff {r-4}(N)$ such that
$$\alpha ^{(n)}(x)=P(f^n(x))\circ \beta ^{(n)}(x)\circ P(x)^{-1}$$
for all $x\in M$ and $n\in \mathbb{N}$. At this section we are going
to show that the image of this $P$ is contained in $\Diff {r-1}(N)$
and also that when $f,\alpha $ and $\beta$ exhibit a higher regularity
so does $P$.\

Let $\tilde{d}$ be the distance in $\textrm{Homeo} (N)$ as defined in
the proof of Lemma \ref{lem:transfer-P-s-u-holon}.

\begin{theorem}
  \label{thm:improving-regularity}
  Let $M$, $N$ and $f$ be as in Theorem~\ref{thm:main-1} and $\alpha
  ,\beta\colon\Z\times M\to\Diff r(N)$ be two $(2r-1)$-dominated
  Lipschitz cocycles over $f$. Let us assume that there exists
  $P\colon M\to\Diff 1(N)$, a Lipschitz continuous map with respect to
  the distance $\tilde{d}$, such that
  \begin{displaymath}
    \alpha^{(1)}(x)=P(f(x))\circ\beta^{(1)}(x)\circ P(x)^{-1},
    \quad\forall x\in M.
  \end{displaymath}
  Assume also that there exists $x\in M$ such that
  $P(x)\in\Diff{r-1}(N)$. Then, $P(M)\subset\Diff{r-1}(N)$. Moreover,
  if $f$ is a $C^{r-1}$ Anosov diffeomorphism and the cocycles
  $\alpha$ and $\beta$ are $C^{r-1}$ then $P$ is $C^{r-1-\varepsilon}$
  for any small $\varepsilon >0$.
\end{theorem}

Recall that the cocycle $\alpha :\mathbb{Z}\times M \to \Diff {r}(N)$
to be $C^{r-1}$ means that the induced map $\alpha^{(1)}\colon M\times
N \to N$ is $C^{r-1}$.

\begin{proof} Note that by Lemma
  \ref{lem:P-restricted-to-su-manifolds}, Remark
  \ref{rem:P-restricted-to-su-manifolds} and Proposition
  \ref{pro:holonomies-existence-regularity} we have automatically that
  $P(W^s(x)\cup W^u(x))\subset \Diff {r-1}(N)$ since $P(x)\in \Diff
  {r-1}(N)$ and $H^{s,\alpha}_{xy}$, $H^{u,\alpha}_{xz}$,
  $H^{s,\beta}_{xy}$ and $H^{u,\beta}_{xz} \in \Diff {r-1}(N)$ for any
  $y\in W^s(x)$ and $z\in W^u(x)$. Applying again this argument we get
  that $P(W^s(y)\cup W^u(y))\subset \Diff {r-1}(N)$ for any $y\in
  W^s(x)\cup W^u(x)$. Now, as any transitive hyperbolic homeomorphism
  is accessible, that is, any two points can be connected by a path
  which is a concatenation of subpaths, where each of which lies
  entirely on a single stable or unstable leaf, it follows by the
  previous arguments that $P(M)\subset \Diff {r-1}(N)$ proving the
  first assertion.\

  Assume now that $f$ is a $C^{r-1}$ Anosov diffeomorphism and that
  $\alpha$ and $\beta$ are $C^{r-1}$ cocycles. Let us consider the
  skew product $F_{\alpha}\colon M\times N\rightarrow M\times N$ given
  by
  \begin{displaymath}
    F_{\alpha}(x,\xi)=(f(x), \alpha ^{(1)}(x)(\xi )), \quad (x,\xi)\in
    M\times N, 
  \end{displaymath}  
  which is a $C^{r-1}$ partially hyperbolic diffeomoprhism since the
  cocycle $\alpha$ is (2r-1)-dominated. We know that the distributions
  $E^s$ and $E^u$ are integrable and the corresponding foliations
  $\hat{W}^s$ (strong stable) and $\hat{W}^u$ (strong unstable)
  respectively, are $C^{r-1}$ foliations, that is, the leaves
  $\hat{W}^s(x,\xi)$ and $\hat{W}^u(x,\xi)$ are $C^{r-1}$ and depend
  continuously on the point $(x,\xi)\in M\times N$ in the $C^{r-1}$
  topology. Moreover, we can observe that they are graphs over
  $W^s(x)$ and $W^u(x)$ respectively. More precisely,
  \begin{center}
    $\hat{W}^s(x,\xi)=\lbrace (y,H^{s,\alpha}_{xy}(\xi)); y\in
    W^s(x)\rbrace$ and $\hat{W}^u(x,\xi)=\lbrace
    (y,H^{s,\alpha}_{xy}(\xi)); y\in W^u(x)\rbrace$.
  \end{center}
  This follows from Theorem 5.5 of \cite{HirshPughShubInvMan}. Now,
  since $\hat{W}^s(x,\xi)$ and $\hat{W}^u(x,\xi)$ are $C^{r-1}$
  submanifolds, it follows that $H^{s,\alpha}_{x}:W^s(x)\times \lbrace
  \xi \rbrace \rightarrow N$ and $H^{u,\alpha}_{x}:W^u(x)\times
  \lbrace \xi \rbrace \rightarrow N$ are $C^{r-1}$.

  Analogously, $H^{s,\beta}_{x}:W^s(x)\times \lbrace \xi \rbrace
  \rightarrow N$ and $H^{u,\beta}_{x}:W^u(x)\times \lbrace \xi \rbrace
  \rightarrow N$ are $C^{r-1}$.

  Combining this with the fact that for every $x\in M$, $y\in W^s(x)$
  and $z\in W^u(x)$ we have $H^{s,\alpha}_{x,y},H^{u,\alpha}_{x,z},
  H^{s,\beta}_{x,y},H^{u,\beta}_{x,z}\in \Diff {r-1}(N) $ (see
  Proposition \ref{pro:holonomies-existence-regularity}) and with
  Lemma \ref{lem:P-restricted-to-su-manifolds} and Remark
  \ref{rem:P-restricted-to-su-manifolds} it follows that $P\mid
  _{W^s(x)}$ and $P\mid _{W^u(x)}$ is $C^{r-1}$ for every $x\in
  M$. Then, applying Journ\'e's Theorem (see \cite{JourneRegLem}) we
  get that $P$ is $C^{r-1-\varepsilon}$ for any small $\varepsilon >0$
  as we want.
 
\end{proof}

To complete the proof of Theorem \ref{thm:main-1} we have just to
combine what we have done so far: in sections
\ref{sec:constr-trans-map} and \ref{sec:concluding-the-proof}, given a
periodic point $x\in M$, we have constructed ``explicitly'' a
Lipschitz continuous map $P:M\rightarrow \Diff {r-4}(N)$ such that
$P(x)=Id$ and
$$\alpha ^{(n)}(x)=P(f^n(x))\circ \beta ^{(n)}(x)\circ P(x)^{-1}$$
for all $x\in M$ and $n\in \mathbb{Z}$. Therefore, by Theorem
\ref{thm:improving-regularity}, $P(M)\subset \Diff {r-1}(N)$ and if
$f$ is a $C^{r-1}$ Anosov diffeomorphism and $\alpha$ and $\beta$ are
$C^{r-1}$ then $P$ is $C^{r-1-\varepsilon}$ for any small $\varepsilon
>0$ which completes the proof of Theorem \ref{thm:main-1}.

\medskip{\bf Acknowledgements.} The first author would like to thank
to Professor Marcelo Viana for useful conversations during the
preparation of this work. L. Backes was supported by CNPq-Brazil. The
second author was partially supported by CNPq-Brazil.

\bibliographystyle{amsalpha} \bibliography{base-biblio}

\end{document}